\documentclass{article}
\usepackage{graphicx} 
\usepackage{amsmath}
\usepackage{amssymb}
\usepackage{amsthm}
\usepackage{hyperref}
\usepackage{xcolor}
\usepackage{a4wide}

\newtheorem{prop}{Proposition}[section]

\newtheorem{thm}[prop]{Theorem}

\newtheorem{rem}[prop]{Remark}

\newcommand{\GL} {\mathop{\mathrm{GL}}}

\def\BF{\mathbb{F}}

\def\BZ{\mathbb{Z}}

\def\GL{\mathrm{GL}}

\title{Drinfeld modular polynomials of level $T$}
\author{Florian Breuer \and Mahefason Heriniaina Razafinjatovo}
\date{\today}

\begin{document}

\maketitle

\begin{abstract}
    We investigate Drinfeld modular polynomials parametrizing $T$-isogenies between Drinfeld $\BF_q[T]$-modules of rank $r\geq 2$. By providing an explicit classification of such isogenies, we derive explicit bounds on the 
    $T$-degrees of the coefficients of the associated modular polynomials. In particular, we obtain exact expressions for the height (i.e. degree of the largest coefficient) of these modular polynomials.
    Numerical computations show that the bounds on the smaller coefficients are often sharp, too.
\end{abstract}





\section{Introduction}

Drinfeld modular polynomials classify isogenies between pairs of Drinfeld modules. The case of Drinfeld modules of rank $r=2$ was first studied in \cite{Bae92} and closely parallels the case of classical modular polynomials parametrizing pairs of elliptic curves linked by a cyclic isogeny of given degree. When $r \geq 3$, the situation is complicated by the fact that now several isomorphism invariants are needed and the resulting modular polynomials have more than two variables, see \cite{BR09, BR16}. 
In all cases, the size of the coefficients (as elements of $\BF_q[T]$) grow very quickly, see e.g. \cite{BPR24, BPR,  CGS20, hsia}.

In this article, we study Drinfeld modular polynomials of rank $r\geq 2$ parametrizing cyclic $T$-isogenies. Such isogenies can be parametrized by a single parameter, allowing a completely explicit description of such polynomials. We exploit this to deduce sharp bounds on the coefficients of such Drinfeld modular polynomials. One may think of this work as a higher rank continuation of \cite{Schweizer}. At the end, we report on some computations showing that, at least in ``small'' cases, these bounds are sharp for many of the coefficients.

\subsection{Isomorphism invariants for rank $r$ Drinfeld modules}

Let $\BF_q$ be the finite field of $q$ elements.
Let $A=\BF_q[T]$ and let $r\geq 2$ be a positive integer. Let $g_1,g_2,\ldots, g_{r-1}, \Delta$ be algebraically independent over $k=\BF_q(T)$.

We define the subring
\[
C \subset A[g_1,\ldots,g_{r-1},\Delta^{-1}]
\]
generated by monomials of the form $a(T)g_1^{e_1}g_2^{e_2}\cdots g_{r-1}^{e_{r-1}}\Delta^{-e_r}$, where 
\[
\sum_{i=1}^{r-1}e_i(q^i-1) = e_r(q^r-1). 
\]
The morphism $C\to A[g_1,\ldots,g_{r-1}]; \Delta\mapsto 1$ is injective, and it is convenient to identify $C$ with its image
\[
C'\subset A[g_1,\ldots,g_{r-1}],
\]
which is generated by monomials $a(T)g_1^{e_1}\cdots g_{r-1}^{e_{r-1}}$ satsifying
\[
\sum_{i=1}^{r-1}e_i(q^i-1) \equiv 0 \bmod (q^r-1).
\]




Then $C$ (and thus also $C'$) is the ring of isomorphism invariants of rank $r$ Drinfeld $A$-modules, as follows. 
Let 
\[
J = \sum_{e_1,\ldots,e_{r-1}} a_{e_1, \ldots, e_{r-1}}(T) g_1^{e_1}\cdots g_{r-1}^{e_{r-1}}\Delta^{-e_r} \in C 
\]
and let $\phi$ be a rank $r$ Drinfeld module defined over an $A$-field $\gamma : A \to L$ by
\[
\phi_T(X) = \gamma(T)X + \ell_1X^q + \cdots + \ell_rX^{q^r}, \qquad \ell_1,\ldots,\ell_r \in L, \ell_r\neq 0.
\]
We evaluate $J$ at $\phi$ by
\[
J(\phi) = \sum_{e_1,\ldots,e_{r-1}} \gamma(a_{e_1, \ldots, e_{r-1}}(T)) \ell_1^{e_1}\cdots \ell_{r-1}^{e_{r-1}}\ell_r^{-e_r} \in L.
\]

Then by \cite{Pot} (see also \cite[Cor. 1.2]{BR16}), two Drinfeld modules $\phi$ and $\phi'$ are isomorphic over an algebraically closed field if and only if $J(\phi) = J(\phi')$ for all $J\in C$.

\subsection{The monic generic Drinfeld module}

We define a special rank $r$ Drinfeld $A$-module $\varphi$ over $K = \BF_q(T,g_1,\ldots,g_{r-1})$ by
\[
\varphi_T(X) = TX + g_1 X^q + \cdots + g_{r-1}X^{q^{r-1}} + X^{q^r}.
\]
We call it the ``monic generic'' rank $r$ Drinfeld module, since every other rank $r$ Drinfeld module $\phi$ defined over an $A$-field $\gamma : A \rightarrow L$ is isomorphic, over a finite extension of $L$, to a specialization of $\varphi$, as follows. 

Suppose $\phi_T(X) = \gamma(T)X + \ell_1X^q + \cdots + \ell_rX^{q^r}$ with $\ell_r\neq 0$. Then $\phi$ is isomorphic to $\phi' = c\phi c^{-1}$ over $L(c)$, where $c^{q^r-1}=\ell_r$, and $\phi'$ is the image of $\varphi$ under the map $K \to L(c);\, T\mapsto \gamma(T),\, g_k \mapsto c^{1-q^k}\ell_k$. We use this $\varphi$ to construct Drinfeld modular polynomials in the next section.

\subsection{Drinfeld modular polynomials}

Let $N\in A$ be monic and choose a basis for $\varphi[N]\cong (A/NA)^r$. 
Let $H\subset (A/NA)^r$ be an $A$-submodule. An isogeny
$f : \varphi \to \varphi^{(f)}$ is said to be of type $H$ if $\ker f\subset \varphi[N]$ is an element of the $\GL_r(A/NA)$-orbit of $H$.

Now for each invariant $J\in C$, we define the Drinfeld modular polynomial of type $H$ for $J$ by
\[
\Phi_{J,H}(X) := \prod_{f : \varphi \to \varphi^{(f)}\;\text{of type $H$}}
\big(X - J(\varphi^{(f)})\big).
\]
It is shown in \cite{BR16} that $\Phi_{J,H}(X)$ has coefficients in $C$, and if we evaluate these coefficients at a Drinfeld module $\psi$, then the roots of the resulting polynomial $\Phi_{J(\psi), H}(X)$ are exactly the $J(\psi^{(f)})$ for isogenies $f : \psi\to\psi^{(f)}$ of type $H$.

In this paper, we consider the two cases $H=(A/TA)$, corresponding to outgoing $T$-isogenies $f: \varphi\to\varphi^{(f)}$, and $H=(A/TA)^{r-1}$, corresponding to incoming $T$-isogenies $f : \varphi^{(f)}\to\varphi$, since each incoming isogeny of type $A/TA$ is dual to an outgoing isogeny of type $(A/TA)^{r-1}$.

The Drinfeld modular polynomials $\Phi_{J,(A/TA)}(X)$ and $\Phi_{J,(A/TA)^{r-1}}(X)$ both have degree 
\[
\psi_r(T) = \frac{q^r-1}{q-1}
\]
in $X$. Our main result is the following sharp estimate on the $T$-degrees of the coefficients of these Drinfeld modular polynomials. 

\begin{thm}\label{Main}
    Let $J = g_1^{e_1}g_2^{e_2}\cdots g_{r-1}^{e_{r-1}}\Delta^{-e_r} \in C$ be a monomial invariant. If $s=1$ or $s=r-1$, then the coefficients of the Drinfeld modular polynomial
    \[
    \Phi_{J, (A/TA)^s}(X) = a_0 + a_1X + \cdots + a_{\psi_r(T)} X^{\psi_r(T)}
    \]
    satisfy
    \begin{align*}
        \deg_T a_0 & = \psi_r(t) w_s(J), \\
        \deg_T a_k & \leq (\psi_r(t)-k) w_s(J),
    \end{align*}
    where
    \[
    w_1(J) :=  q\left( \sum_{i=1}^{r-1}e_i - e_r\right) 
    \quad\text{and}\quad
    w_{r-1}(J) := \sum_{i=1}^{r-1}e_i + e_r(q^{r}-q^{r-1}-1).
    \]
    In particular, the {\em height} ($T$-degree of the largest coefficient) of $\Phi_{J, (A/TA)^s}(X)$ satisfies
    \[
    h(\Phi_{T,(A/TA)^s}) = \psi_r(T)w_s(J). 
    \]
\end{thm}

It is instructive to compare this with existing results in the literature. When $r=2$, $J = j = g_1^{q+1}\Delta^{-1}$ we have $w_1(J)=q^2$ and $\psi_2(T)=q+1$. Now Theorem \ref{Main} tells us that the height of $\Phi_{J,(A/TA)}(X) =: \Phi_T(X,j)$ equals $q^2(q+1)$. This confirms \cite[Prop. 8]{Schweizer}. In \cite{BPR24} the following expression is obtained (we restrict to $N=T$ here):
\[
 h(\Phi_T) = \frac{q^2-1}{2}\psi_2(T)\big(\deg T - 2\lambda_T + b_q(T)\big) = \frac{q^2-1}{2}(q+1)\left(1 - \frac{2}{q+1} +b_q(T)\right).
\]
In this case, we find $b_q(T) = \frac{q^2+2q-1}{q^2-1}$, which is well within the bounds stipulated in \cite[Thm 1.1]{BPR24}.

\section{$T$-Isogenies}

In this section we classify all incoming and outgoing $T$-isogenies of a given Drinfeld module. 
Let $L$ be an $A$-field of generic characteristic, i.e. one where $\gamma  :A \hookrightarrow L$. 
Suppose two Drinfeld modules over $L$ are given by
\[
\phi_T(X) = TX + g_1X^q + \cdots + g_{r-1}X^{q^{r-1}} + \Delta X^{q^r},
\]
\[
\tilde{\phi}_T(X) = TX + \tilde{g}_1X^q + \cdots + \tilde{g}_{r-1}X^{q^{r-1}} + \tilde{\Delta}X^{q^r}.
\]

    

\begin{prop}\label{Main1}
Fix a Drinfeld module $\phi$ as above.
\begin{enumerate}

    \item 
    Consider the polynomial
    \[
    Q(x) = \sum_{i=0}^r(-1)^i g_{r-i} x^{\frac{q^r-q^{r-i}}{q-1}} \in L[x]
    \]
    with the conventions $g_0=T$ and $g_r=\Delta$.
    Then, up to isomorphism, all outgoing $T$-isogenies $f : \phi \to \tilde{\phi}$ are given by
    \[
    f(X) = X + aX^q
    \]
    where $a$ ranges over the roots of $Q(x)$. 
    Furthermore, for each such $a$,  the coefficients of $\tilde{\phi}$ are given by
    \begin{align*}
    & \tilde{g}_k = g_k + a g_{k-1}^q - a^{q^{k-1}}\tilde{g}_{k-1}, \quad k=1,\ldots,r-1 \\
    & \tilde{\Delta} = a^{1-q^r}\Delta^q,
    \end{align*}
    with the conventions $g_0 = \tilde{g}_0 = T$.

    \item
    Consider the polynomial 
    \[
    \tilde{Q}(x) := \sum_{i=0}^r(-1)^ig_{r-i}^{q^i}x^{\frac{q^i-1}{q-1}} \in L[x]
    \]
    with the conventions $g_0=T$ and $g_r=\Delta$.
    
    Then, up to isomorphism, all incoming $T$-isogenies $f : \tilde{\phi} \to \phi$ are given by
    \[
    f(X) = X + aX^q
    \]
    where $a$ ranges over the roots of $\tilde{Q}(x)$. 
    Furthermore, for each such $a$,  the coefficients of $\tilde{\phi}$ are given by
    \begin{align*}
    & \tilde{g}_k = g_k + a^{q^{k-1}}g_{k-1}  - a\tilde{g}_{k-1}^q, \quad k=1,\ldots,r-1 \\
    & \tilde{\Delta}^q = a^{q^r-1}\Delta,
    \end{align*}
    with the conventions $g_0 = \tilde{g}_0 = T$.

\end{enumerate}
\end{prop}

\begin{proof}
    To prove Part 1, we fix $\phi$ and consider an outgoing $T$-isogeny
    $f : \phi \to \tilde{\phi}$.
    Then $f(X)$ is $\BF_q$-linear of degree $q$ and up to isomorphism, we may assume the linear term is $X$, so $f(X) = X + aX^q$. 

    There exists a dual isogeny $\hat{f} : \tilde\phi \to \phi$,
    \[
    \hat{f}(X) = b_0X + b_1X^q + \cdots + b_{r-1}X^{q^{r-1}},
    \]
    for which 
    \begin{equation}\label{eq:dual}
        \hat{f}\circ f(X) = \phi_T(X).
    \end{equation}

    We compare coefficients in (\ref{eq:dual}) and obtain $b_0=T$ along with
    \[
    b_k = g_k - b_{k-1}a^{q^{k-1}}, \quad k=1,2,\ldots,r,
    \]
    where we set $b_r=0$ and $g_r=\Delta$. By induction on $k$, it follows that
    \[
    b_k = \sum_{i=0}^k (-1)^i g_{k-i} a^{\frac{q^k-q^{k-i}}{q-1}}.
    \]
    Now the $k=r$ case is just
    \[
    Q(a)=0,
    \]
    as required. 

    Every $T$-isogeny $f : \phi\rightarrow \tilde{\phi}$ gives rise to a root $a$ of $Q(x)$, and different roots $a$ correspond to different factorisations of $\phi_T(X) = \hat{f}\circ f(X)$, hence correspond to non-isomorphic isogenies.
    
    Conversely, the number of $T$-isogenies up to isomorphism is $\psi_r(T) = \deg_x(Q)$, so every root of $Q(x)$ corresponds to such an isogeny.

    Finally, for a given isogeny $f(X)=X+aX^q$, we obtain the stated relations between the coefficients of $\phi$ and $\tilde{\phi}$ by comparing coefficients in 
    $f\circ\phi_T(X) = \tilde{\phi}_T\circ f(X)$.
    This completes the proof of Part 1.

    To prove Part 2, we fix $\phi$ and suppose $f : \tilde{\phi}\to\phi$ is an incoming $T$-isogeny.
    Then $f(X)$ is $\BF_q$-linear of degree $q$ and up to isomorphism, we may assume the linear term is $X$, so $f(X) = X + aX^q$. 

    We compare coefficients in 
    \begin{equation}
        \label{eq:dual2}
        f\circ \hat{f}(X) = \phi_T(X),
    \end{equation}
    where 
    \[
    \hat{f}(X) = b_0X + b_1X^q + \cdots + b_{r-1}X^{q^{r-1}}
    \]
    is a dual isogeny, and obtain $b_0=T$ and 
    \[
    b_k = g_k - ab_{k-1}^q, \quad k=1,2,\ldots,r,
    \]
    with the conventions $b_r=0$ and $g_r=\Delta$.
    By induction on $k$, it follows that 
    \[
    b_k = \sum_{i=0}^k(-1)^i g_{k-i}^{q^i} a^{\frac{q^i-1}{q-1}}.
    \]
    Again, the case $k=r$ is simply $0 = Q(a),$
    as required.

    The same argument as in Part 1 now concludes the proof.
\end{proof}

We record a variant of Proposition \ref{Main1}, which is more useful when computing Drinfeld modular polynomials.

\begin{prop}\label{altisog}
    Fix a Drinfeld module $\phi$ as above.
\begin{enumerate}

    \item 
    Consider the polynomial
    \[
    Q(x) = \sum_{i=0}^r(-1)^i g_{r-i} x^{\frac{q^r-q^{r-i}}{q-1}} \in L[x]
    \]
    with the conventions $g_0=T$ and $g_r=\Delta$.
    Then, up to isomorphism, all outgoing $T$-isogenies $f : \phi \to \tilde{\phi}$ are given by
    \[
    f(X) = a^{-1}X + X^q
    \]
    where $a$ ranges over the roots of $Q(x)$. 
    Furthermore, for each such $a$,  the coefficients of $\tilde{\phi}$ are given by
    \begin{align*}
    & \tilde{g}_k = a^{q^k}(a^{-1}g_k + g_{k-1}^q - \tilde{g}_{k-1}),
    \quad k=1,\ldots,r-1 \\
    & \tilde{\Delta} = \Delta^q,
    \end{align*}
    with the conventions $g_0 = \tilde{g}_0 = T$.

    \item
    Consider the polynomial 
    \[
    \tilde{Q}(x) := \sum_{i=0}^r(-1)^ig_{r-i}^{q^{i-1}}x^{\frac{q^i-1}{q-1}} \in L[x]
    \]
    with the conventions $g_0=T$ and $g_r=\Delta$.
    
    Then, up to isomorphism, all incoming $T$-isogenies $f : \tilde{\phi} \to \phi$ are given by
    \[
    f(X) = a^{-1}X + X^q
    \]
    where $a$ ranges over the roots of $\tilde{Q}(x)$. 
    Furthermore, for each such $a$,  the coefficients of $\tilde{\phi}$ are given by
    \begin{align*}
    & \tilde{g}_k = a(g_{k-1} - \tilde{g}_{k-1}^q) + a^{1-q^k}g_k,
    \quad k=1,\ldots,r-1 \\
    & \tilde{\Delta}^q = \Delta,
    \end{align*}
    with the conventions $g_0 = \tilde{g}_0 = T$.

\end{enumerate}
\end{prop}
    
\begin{proof}
    The proof is similar to that of Proposition \ref{Main1}.
\end{proof}

\section{Bounding the coefficients}


In this section, we apply our description of $T$-isogenies (Proposition \ref{Main1}, which is easier to use here than Proposition \ref{altisog}) to estimate the $T$-degrees of the coefficients of our Drinfeld modular polynomials.

Let $F = \BF_q(g_1,g_2,\ldots,g_{r-1})$ and consider $K = \BF_q(T, g_1, \ldots, g_{r-1})$ as a rational function field of transcendence degree one over $F$. 
We endow $K$ with the valuation $v : K\to\BZ$ with uniformizer $\frac{1}{T}$, i.e. 
$v(T) = -1$ and $v(x) = 0$ for all $x\in F$. 

Let $K_T$ denote the splitting field of $\varphi_T(X)$ over $K$
and extend the valuation $v$ to any valuation of $K_T$.
By \cite[\S2]{BR16} the roots of 
\[
\Phi_{J(A/TA)}(X), \Phi_{J,(A/TA)^{r-1}}(X) \in K[X]
\]
lie in $K_T$. Since the $T$-degree of each coefficient equals the negative of its valuation, our approach is to prove Theorem \ref{Main} by computing the valuation of the roots of 
these Drinfeld modular polynomials.




\begin{proof}[Proof of Theorem \ref{Main}]

We first prove the case $s=1$. 
Applying Proposition \ref{Main1} to the $T$-isogenies $f : \varphi \to \varphi^{(f)}$, we see that each isogeny corresponds to a root $a$ of 
\[
Q(x) = 1 + \left(\sum_{i=1}^{r-1}(-1)^i g_{r-i} x^{\frac{q^r-q^{r-i}}{q-1}}\right) + (-1)^r T x^{\frac{q^r-1}{q-1}} \in K[x].
\]
We extend the valuation $v$ to any valuation on the splitting field of $Q(x)$. 
The Newton polygon of $Q(x)$ consists of a single line segment from the origin to the point $\left(\frac{q^r-1}{r-1}, -1\right)$. It follows that each root $a$ has valuation
\[
v(a) = \frac{q-1}{q^r-1}.
\]
By Proposition \ref{Main1}, the coefficients of the corresponding target Drinfeld module $\tilde{\phi}=\varphi^{(f)}$ satisfy
\[
\tilde{g}_1 = g_1 + a(T^q-T)
\]
\[
\tilde{g}_k = g_k + ag_{k-1}^q - a^{q^{k-1}}\tilde{g}_{k-1}, \quad  k=2,3,\ldots,r.
\]
In particular, we see by induction on $k$ that 
\[
v(\tilde{g}_k) = \frac{q^k-1}{q^r-1} - q < 0\quad\text{for}\quad k=1,2,\ldots,r.
\]


Since $\sum_{i=1}^{r-1}e_i(q^i-1) = e_r(q^r-1)$ we compute:
\begin{align*}
v(J(\varphi^{(f)})) & = \sum_{i=1}^{r-1}e_i v(\tilde{g}_i) - e_r v(\tilde{\Delta}) \\
& = \sum_{i=1}^{r-1}e_i \left(\frac{q^i-1}{q-1} - q\right) + e_r(q-1) \\
& = -q\left( \sum_{i=1}^{r-1}e_i - e_r\right) = -w_1(J).
\end{align*}

Since each coefficient $a_i$ of $X^i$ in $\Phi_{J,(A/TA)}(X)$ is an elementary symmetric polynomial of degree $\psi_r(T) - i$ in the roots $J(\varphi^{(f)})$, the estimate for $s=1$ follows.

\medskip

To prove the case $s=r-1$, we perform a similar calculation, this time for the incoming $T$-isogenies $f : \varphi^{(f)} \to \varphi$. By Proposition \ref{Main1}, these correspond to roots $a$ of 
\[
\tilde{Q}(x) = 1 + \left(\sum_{i=1}^{r-1}(-1)^i g_{r-i}^{q^i}x^{\frac{q^i-1}{q-1}}\right) + (-1)^r T^{q^r} x^{\frac{q^r-1}{q-1}} \in K[x].
\]
We again extend the valuation $v$ to the splitting field of $\tilde{Q}(x)$.
The Newton polygon of $\tilde{Q}(x)$ consists of a single line segment from the origin to the point $\left(\frac{q^r-1}{q-1}, -q^r\right)$, hence each root $a$ has valuation
\[
v(a) = q^r\left(\frac{q-1}{q^r-1}\right).
\]
Using Proposition \ref{Main1}, we find again by induction on $k$ that the coefficients of the source Drinfeld module $\tilde{\phi} = \varphi^{(f)}$ satisfy
\begin{align*}
v(\tilde{g}_k) & = \frac{q^k-1}{q^r-1} - 1 < 0, \quad k=1,2,\ldots, r-1, \\
v(\tilde{\Delta}) & =  q^{r-1}(q-1).
\end{align*}

Again, we compute
\begin{align*}
v(J(\tilde{\phi})) & = \sum_{i=1}^{r-1}e_i v(\tilde{g}_i) - e_rv(\tilde{\Delta}) \\
& = \sum_{i=1}^{r-1}e_i \left(\frac{q^i-1}{q^r-1} - 1\right) - e_r q^{r-1}(q-1) \\
& = -\sum_{i=1}^{r-1}e_i - e_r(q^r-q^{r-1}-1) = -w_{r-1}(J).
\end{align*}
The result now follows.
\end{proof}

\begin{rem}
    One may similarly obtain bounds on the $g_i$-degrees of the terms in $\Phi_{J,(A/TA)^s}(X)$ by considering a valuation $v_i$ with uniformizer $1/g_i$. 
    However, the resulting bounds are more complicated and less sharp. This may be a topic for future work. Similar bounds are obtained in \cite{BR09} via analytic methods.  
\end{rem}

\section{Examples}

In this section we compute some examples of Drinfeld modular polynomials and verify the bounds in Theorem \ref{Main}.

Let $J = g_1^{e_1}\cdots g_{r-1}^{e_{r-1}} \Delta^{-e_r} \in C$ be a monomial invariant. It will be more convenient to work with the isomorphic ring $C'$ of invariants, where $\Delta=1$ and $J = g_1^{e_1}\cdots g_{r-1}^{e_{r-1}}$.

To compute the Drinfeld modular polynomials $\Phi_{J, A/TA}(X)$, we follow the approach of \cite{BR09}. By Proposition \ref{altisog}, the isogenies $f : \varphi \rightarrow \varphi^{(f)}$ are given by $f(X) = a^{-1}X + X^q$, where $a$ ranges over the roots of $Q(x)$. It will be convenient to instead consider the roots $v = a^{-1}$ of the reciprocal polynomial 
\[
P(x) = x^{\psi_r(T)}Q(x^{-1}).
\]
Let $M_v$ be an $r\times r$ matrix over $K$ with characteristic polynomial $P(x)$ and compute its inverse $M_a = M_v^{-1}$. This plays the role of our roots $a$ (more precisely, its eigenvalues are the $a$'s).

Next, compute the matrices $M_{\tilde{g_i}}$ corresponding to the coefficients of the target Drinfeld module by substituting the matrix $M_a$ into $a$ in the expressions for $\tilde{g}_i$ given in Proposition \ref{altisog}.

Finally, our polynomial $\Phi_{J, A/TA}(X)$ is the characteristic polynomial of the matrix $M_{\tilde{J}}$ corresponding to the invariant $\tilde{J} = J(\phi^{(f)})$:
\begin{align*}
M_{\tilde{J}} & = M_{\tilde{g}_1}^{e_1}\cdots M_{\tilde{g}_{r-1}}^{e_{r-1}}, \\
\Phi_{J, A/TA}(X) & = \det(IX - M_{\tilde{J}}) \\
& = a_0 + a_1X + \cdots + a_{\psi_r(T)-1}X^{\psi_{r-1}(T)} + X^{\psi_r(T)} \in C'[X].
\end{align*} 
The polynomial $\Phi_{J, (A/TA)^{r-1}}(X)$ is computed similarly, this time using $\tilde{Q}(x)$ instead of $Q(x)$. Note that the constant coefficient of $\tilde{Q}(x)$ is $\Delta^{1/q} = 1$.


The easiest coefficient to compute is $a_{\psi_r(T)-1} = -\mathrm{Tr}\big(M_{\tilde{J}}\big)$; its degree is bounded by the weight $w_1(J)$ or $w_{r-1}(J)$ from Theorem \ref{Main}. Computing this first and checking that the result is indeed in $C'[X]$ is a valuable sanity check before attempting to compute the complete polynomial.

This approach via the characteristic polynomial has the advantage of being conceptually simple and easy to implement. It is, however, very slow, with both the time and space requirements growing rapidly with $\psi_r(T)$ and $w_s(J)$. The ``small'' cases reported below were computed on a mid-level (2024 era) gaming laptop  using Pari/GP \cite{PARI2} and took at most a few hours each.
The code can be found in \cite{Github}. 

\subsection{The case $r=3$, $q=2$.}

When $r=3$ and $A=\BF_2[T]$, we have $\psi_3(t) = 7$ and invariants are of the form
\[
J_{e_1\, e_2} = g_1^{e_1}g_2^{e_2}\Delta^{-e_3},\quad\text{where}\quad
e_1 + 3e_2 = 7e_3.
\]
The relevant weights are
\[
w_1(J_{e_1\, e_2}) = 2(e_1 + e_2 - e_3) 
\quad\text{and}\quad
w_2(J_{e_1\, e_2}) = e_1 + e_2 + 3e_3.
\]

It it shown in \cite{BR09} that the four invariants in the following table generate the ring $C$ of invariants:

\medskip

\begin{tabular}{|c|c|c|c|c|c|}\hline
    $J$ & $e_1$ & $e_2$ & $e_3$ & $w_1(J)$ & $w_2(J)$ \\ \hline 
    $J_{12}$ & 1 & 2 & 1 & 4 & 6 \\
    $J_{41}$ & 4 & 1 & 1 & 8 & 8 \\
    $J_{70}$ & 7 & 0 & 1 & 12 & 10 \\
    $J_{07}$ & 0 & 7 & 3 & 8 & 16 \\ \hline
\end{tabular}

\medskip 
The polynomials $\Phi_{J, (A/TA)^2}(X)$ for these invariants were computed in \cite{BR09} and one finds that all of the degree bounds from Theorem \ref{Main} are in fact equalities.

We similarly computed the polynomials $\Phi_{J, (A/TA)}(X)$ and found again that all bounds in Theorem \ref{Main} are sharp for these four invariants.

\subsection{The case $r=4$, $q=2$.}

In this case our modular polynomials have degree $\psi_2(T) = 15$. 
Invariants are given by
\[
J_{e_1\, e_2\, e_3} = g_1^{e_1}g_2^{e_2}g_3^{e_3}\Delta^{-e_4},
\quad\text{where}\quad
e_1 + 3e_2 + 7e_3 = 15e_4.
\]
The relevant weights are
\[
w_1(J_{e_1\, e_2\, e_3}) = 2(e_1 + e_2 + e_3 - e_4) 
\quad\text{and}\quad
w_3(J_{e_1\, e_2\, e_3}) = e_1 + e_2 + e_3 + 7e_4.
\]

We list some small invariants in $C$:

\medskip

\begin{tabular}{|c|c|c|c|c|c|c|}\hline
    $J$ & $e_1$ & $e_2$ & $e_3$ & $e_4$ & $w_1(J)$ & $w_3(J)$ \\ \hline 
    $J_{102}$ & 1 & 0 & 2 & 1 & 4 & 10 \\
    $J_{050}$ & 0 & 5 & 0 & 1 & 8 & 12 \\
    $J_{221}$ & 2 & 2 & 1 & 1 & 8 & 12 \\ 
    \hline
\end{tabular}

\medskip

Once again, we find that the bounds in Theorem \ref{Main} are sharp in the cases of $\Phi_{J_{102}, A/TA}(X),$ $\Phi_{J_{221}, A/TA}(X),$
$ \Phi_{J_{050},(A/TA)}(X)$ and $\Phi_{J_{102}, (A/TA)^3}(X)$.

\subsection{The case $r=3$, $q=3$.}

In this case our modular polynomials have degree $\psi_3(T) = 13$. We list some small invariants in $C$:

\medskip

\begin{tabular}{|c|c|c|c|c|c|}\hline
    $J$ & $e_1$ & $e_2$ & $e_3$ & $w_1(J)$ & $w_2(J)$ \\ \hline 
    $J_{13}$ & 1 & 3 & 1 & 9 & 21 \\
    $J_{52}$ & 5 & 2 & 1 & 18 & 24 \\
    $J_{91}$ & 9 & 1 & 1 & 27 & 27 \\
    $J_{26}$ & 2 & 6 & 2 &18 & 42 \\ \hline
\end{tabular}

\medskip

In this case, the degree bounds from Theorem \ref{Main} are not all sharp:
\medskip

{\small
\begin{tabular}{|c||c|c||c|c||c|c|} \hline 
    & \multicolumn{2}{|c||}{$\Phi_{J_{13}, (A/TA)}(X)$} & 
    \multicolumn{2}{|c||}{$\Phi_{J_{52}, (A/TA)}(X)$} & 
    \multicolumn{2}{|c|}{$\Phi_{J_{26}, (A/TA)}(X)$}\\ \hline 
     $i$ & $\deg_T(a_i)$ & $9(13 - i)$ & $\deg_T(a_i)$ & $18(13-i)$ & $\deg_T(a_i)$ & $18(13-i)$ \\ \hline
     0 & 117 & 117 & 234 & 234 & 234 & 234\\
    1 & 108 & 108 & 216 & 216 & 216 & 216\\
    2 & \bf{94} & 99 & \bf{193} & 198 & \bf{193} & 198\\
    3 & 90 & 90 & 180 & 180 & 180 & 180\\
    4 & 81 & 81 & 162 & 162 & 162 & 162\\
    5 & \bf{58} & 72 & \bf{133} & 144 & \bf{130} & 144\\
    6 & \bf{49} & 63 & \bf{115} & 126 & \bf{112} & 126\\ 
    7 & \bf{40} & 54 & \bf{97} & 108 & \bf{94} & 108\\
    8 & \bf{31} & 45 & \bf{79} & 90 & \bf{76} & 90\\
    9 & 36 & 36 & 72 & 72 & 72 & 72\\
    10 & 27 & 27 & 54 & 54 & 54 & 54\\
    11 & \bf{13} & 18 & \bf{31} & 36 & \bf{31} & 36\\
    12 & 9 & 9 & 18 & 18 & 18 & 18\\ \hline 
\end{tabular}
}

\medskip

\begin{tabular}{|c||c|c||c|c|} \hline 
    & \multicolumn{2}{|c||}{$\Phi_{J_{13}, (A/TA)^2}(X)$} 
    & \multicolumn{2}{|c|}{$\Phi_{J_{52}, (A/TA)^2}(X)$} \\ \hline 
    $i$ & $\deg_T(a_i)$ & $21(13 - i)$ 
    & $\deg_T(a_i)$ & $24(13 - i)$\\ \hline 
    0 & 273 & 273 & 312 & 312\\
    1 & 252 & 252 & 288 & 288\\ 
    2 & \bf{228} & 231 & \bf{255} & 264\\
    3 & 210 & 210 & 240 & 240\\
    4 & 189 & 189 & 216 & 216\\
    5 & \bf{159} & 168 & \bf{180} & 192\\
    6 & \bf{138} & 147 & \bf{159} & 168\\
    7 & \bf{117} & 126 & \bf{132} & 144\\
    8 & \bf{96} & 105 & \bf{108} & 120\\
    9 & 84 & 84 & 96 & 96\\
    10 & 63 & 63 & 72 & 72\\
    11 & \bf{39} & 42 & \bf{39} & 48\\
    12 & 21 & 21 & 24 & 24\\ \hline
\end{tabular}

\medskip

It is very curious that in all five cases above, the coefficients where the bound fails to be sharp are $a_{2}, a_{5}, a_{6}, a_{13-6}, a_{13-5}$ and $a_{13-2}$.

\paragraph{Acknowledgements.} The first author was supported by the Alexander-von-Humboldt Foundation. We thank Fabien Pazuki for helpful comments.

\bibliographystyle{amsplain}
\bibliography{ModPolyT} 

\noindent
\textbf{Florian Breuer} \\
School of Mathematical Sciences \\
University of Newcastle, Australia. \\
\href{mailto:florian.breuer@newcastle.edu.au}{florian.breuer@newcastle.edu.au}

\vspace{1em}

\noindent
\textbf{Heriniaina Razafinjatovo} \\
1) University of Antananarivo, Antananarivo, Madagascar. \\
\href{mailto:heriniaina.razafinjatovo@univ-antananarivo.mg}{heriniaina.razafinjatovo@univ-antananarivo.mg} \\
2) iTUNIVERSITY, Andoharanofotsy, Madagascar. \\
\href{mailto:heriniaina.razafinjatovo@ituniversity-mg.com}{heriniaina.razafinjatovo@ituniversity-mg.com}

\end{document}